\documentclass[12pt]{article}
\usepackage{amsmath,amsthm,amssymb}
\usepackage[english]{babel} 
\usepackage{array}
\usepackage{rotating}
\usepackage{graphicx}
\usepackage{comment}
\usepackage{pdflscape}
\usepackage{url}
\usepackage[latin1]{inputenc}
\usepackage{tikz}
\usetikzlibrary{shapes,arrows}
\usepackage{textcomp}
\usetikzlibrary{calc}
\usepackage{latexsym}
\usepackage{setspace}

\newcommand\E{{\mathbb E}}

\newcommand\almostsure{\buildrel {\rm a.s.}\over\longrightarrow}

\newcommand\matA{{\bf A}}

\newcommand\vecD{{\bf D}}
\newcommand\vecX{{\bf X}}
\newcommand\vecv{{\bf v}}

\newcommand\normal{{\cal N}}

\newtheorem{lem}{Lemma}[section]
\newtheorem{cor}{Corollary}[section]
\newtheorem{theorem}{Theorem}[section]

\newcommand\convD{{\buildrel {\rm d} \over \longrightarrow}}

\newcommand\Polya{P\' olya}

\begin{document}
\begin{center}
{\Large \bf Degree profile of $m$--ary search trees:\\
A vehicle for data structure compression 

\bigskip
\large\bf Ravi Kalpathy\footnote{Department of Mathematics and Statistics, University of Massachusetts Amherst, Massachusetts~01003, USA. Email: rkalpathy@math.umass.edu} \quad \quad Hosam Mahmoud\footnote{Department of
Statistics, The George Washington University, Washington,
DC~20052, USA. Email: hosam@gwu.edu}}

\end{center}
\section*{Abstract.}
We revisit the random $m$--ary search 
tree and study a finer profile of its node outdegrees with the purpose of exploring possibilities of data structure compression. 
The analysis is done via \Polya\ urns. 
The analysis shows that the number of nodes of 
each individual node outdegree has a phase transition:
Up to $m=26$, the number of nodes of outdegree $k$, for $k=0,1, \ldots, m$,
is asymptotically normal; that behavior changes at $m = 27$.
Based on the analysis, we propose a compact $m$-ary tree  
that offers significant space saving.

\bigskip\noindent
 Keywords: 
Random structure, algorithm, data structure compression,  
\Polya\ urn, Normal limit law, phase transition. 

\bigskip\noindent
MSC: 
60C05,     
68W40     
\section{Introduction}
The $m$--ary search tree is a fundamental branching structure 
that models algorithms
and data structures. For example, the binary search tree serves as a backbone for the analysis of certain searching and sorting algorithms 
(see~\cite{Knuth,Evolutionbook,Sortingbook}). For larger values of $m$,
the $m$--ary search tree
is popular in database applications and is used as a mathematical abstraction for hierarchical data storage (particularly balanced versions, like the $B$--tree, 
see~\cite{Rivest}; Chapter 18).
It is our aim in this article to revisit $m$--ary search 
trees and study a finer profile of node outdegrees, 
with the purpose of exploring possibilities of data structure compression. 
The method will be modeling via \Polya\ urns.
A genesis of these ideas is in~\cite{Devroye}, which conducts a complete
analysis in the binary case.
The urn proposed works ``bottom-up," meaning a model that colors insertion positions that lie on the fringe of the tree. Other possible models
may go via generalized \Polya\ urns, as was done in the recent study
of protected nodes in ternary search trees~\cite{Cecilia},
where the balls of different colors  in the urn 
have various levels of ``activity." In this generalized model the tree is chopped into a large number of shrubs, many of which are not at the bottom (some even include the root), and the shrubs are colored.
\section{Random $m$--ary search trees} 
The $m$-ary search tree is a tree with nodes of outdegrees 
restricted to be at most $m\ge 2$. Each node holds data (called
``keys" in the jargon). A node holds up to $m-1$ keys from an ordered 
domain, such as the real numbers with the usual ordinal relations among them. The keys in a node are kept in sorted
order, say they are arranged from left to right in ascending order. 
If we have $n\le m-1$ keys, they are kept in sorted order in
a root node, which is
a single node in the tree (i.e., it is the root).
If we have $n> m-1$ keys, the first 
$m-1$ among  them go into the root node of the tree (and are stored in sorted order); suppose they are $X_1, \ldots, X_{m-1}$ and after sorting they are 
$X_{(1)} \le X_{(2)} \le \ldots \le X_{(m-1)}$. 
Such order statistics of $m-1$ keys create an $m$-chotomy of the data
domain. For example, if these $m-1$ keys are distinct 
real numbers, they cut the real line
into $m$ segments. 
Subsequent keys are classified according to
their relative order to those in the root.
All keys that fall in the data interval $[X_{(i)}, X_{(i+1)})$ go into the $(i+1)$st subtree,
for $i=0, 1, \ldots, m-1$
(taking $X_{(0)} = -\infty$ and $X_{(m)} = \infty$). Recursively,
a key directed to the $i$th subtree
is subjected to the same insertion algorithm. That is to say, 
an $m$--ary tree is either empty, consists of a single root containing up to $m-1$ sorted keys,  or has a root containing $m-1$ sorted keys and $m$ ordered subtrees that are themselves $m$--ary search trees, with all the keys in the $i$th subtree falling between the $(i-1)$st and $i$th keys of the root.

The standard probability model on data assumes the keys to be $n$ 
real numbers sampled from a common continuous probability distribution, 
or equivalently, their ranks (almost surely) form a random permutation 
of the set $\{1,\ldots, n\}$ (all $n!$ permutations are equally likely); 
see~\cite{Knuth}. For the purpose of analysis,
we can assimilate the data by their ranks.
Figure~\ref{Fig:mary} shows a quaternary tree (4-way tree; $m$--ary tree with $m= 4$) constructed from the permutation $(12, 16,11, 9,  13, 7, 3, 5, 15, 1, 4, 14, 10, 8, 2, 6)$.

\usetikzlibrary{calc,intersections}

\begin{figure}[thb]
\begin{center}
\begin{tikzpicture}
\coordinate (A) at (0,0);
\coordinate (B) at (1.8,0);
\coordinate (C) at (1.8,0.6);
\coordinate (D) at (0,0.6);
\draw [thick] (A)--(B)--(C)--(D)--cycle;

\node[draw=white] at (0.3, 0.3) {$11$};
\node[draw=white] at (0.9, 0.3) {$12$};
\node[draw=white] at (1.5, 0.3) {$16$};

\coordinate (E) at (0.6, 0);
\coordinate (F) at (0.6, 0.6);
\draw [thick] (E)--(F);

\coordinate (G) at (1.2, 0);
\coordinate (H) at (1.2, 0.6);
\draw [thick] (G)--(H);

\coordinate (I) at (-0.3,0);
\coordinate (J) at (2.1,0);
\coordinate (K) at (2.1,-0.6);
\coordinate (L) at (-0.3,-0.6);
\draw [thick] (I)--(J)--(K)--(L)--cycle;

\coordinate (M) at (0.3, 0);
\coordinate (N) at (0.3, -0.6);
\draw [thick]  (M)--(N);

\coordinate (O) at (0.9, 0);
\coordinate (P) at (0.9, -0.6);
\draw [thick] [thick]  (O)--(P);
\draw (M)--(P);
\draw (N)--(O);

\coordinate (Q) at (1.5, 0);
\coordinate (R) at (1.5, -0.6);
\draw [thick]  (Q)--(R);

\draw (J)--(R);
\draw (K)--(Q);


\coordinate (A) at (-2.0,-2);
\coordinate (B) at (-0.2,-2);
\coordinate (C) at (-0.2,-1.4);
\coordinate (D) at (-2.0,-1.4);
\draw [thick] (A)--(B)--(C)--(D)--cycle;

\coordinate (I) at (-2.3,-2);
\coordinate (J) at (0.1,-2);
\coordinate (K) at (0.1,-2.6);
\coordinate (L) at (-2.3,-2.6);
\draw [thick] (I)--(J)--(K)--(L)--cycle;

\node[draw=white] at (-1.7, -1.7) {$3$};
\node[draw=white] at (-1.1, -1.7) {$7$};
\node[draw=white] at (-0.5, -1.7) {$9$};

\coordinate (E) at (-1.4, -2);
\coordinate (F) at (-1.4, -1.4);
\draw [thick] (E)--(F);

\coordinate (G) at (-0.8, -2);
\coordinate (H) at (-0.8 ,-1.4);
\draw [thick]  (G)--(H);

\coordinate (M) at (-1.7, -2.6);
\coordinate (N) at (-1.7, -2.0);
\draw [thick]  (M)--(N);

\coordinate (O) at (-1.1, -2.6);
\coordinate (P) at (-1.1, -2.0);
\draw [thick] (O)--(P);

\coordinate (Q) at (-0.5, -2.6);
\coordinate (R) at (-0.5, -2.0);
\draw [thick] (Q)--(R);


pointers from the node
\coordinate (C1) at (-2, -2.3);
\coordinate (C2) at (-5.6, -3.9);
\draw [thick] (C1)--(C2);

\coordinate (C3) at (-1.4, -2.3);
\coordinate (C4) at (-2.6, -3.9);
\draw [thick] (C3)--(C4);

\coordinate (C5) at (-0.8, -2.3);
\coordinate (C6) at (0.4, -3.9);
\draw [thick] (C5)--(C6);

\coordinate (C7) at (-0.2, -2.3);
\coordinate (C8) at (3.4, -3.9);
\draw [thick]  (C7)--(C8);


\coordinate (A) at (1.0,-2);
\coordinate (B) at (2.8,-2);
\coordinate (C) at (2.8,-1.4);
\coordinate (D) at (1.0,-1.4);
\draw [thick] (A)--(B)--(C)--(D)--cycle;

\coordinate (I) at (0.7,-2);
\coordinate (J) at (3.1,-2);
\coordinate (K) at (3.1,-2.6);
\coordinate (L) at (0.7,-2.6);
\draw [thick] (I)--(J)--(K)--(L)--cycle;

\node[draw=white] at (1.3, -1.7) {$13$};
\node[draw=white] at (1.9, -1.7) {$14$};
\node[draw=white] at (2.5, -1.7) {$15$};

\coordinate (E) at (1.6, -2);
\coordinate (F) at (1.6, -1.4);
\draw [thick] (E)--(F);

\coordinate (G) at (2.2, -2);
\coordinate (H) at (2.2 ,-1.4);
\draw [thick]  (G)--(H);

\coordinate (M) at (1.3, -2.6);
\coordinate (N) at (1.3, -2.0);
\draw [thick]  (M)--(N);

\coordinate (O) at (1.9, -2.6);
\coordinate (P) at (1.9, -2.0);
\draw [thick]  (O)--(P);

\coordinate (Q) at (2.5, -2.6);
\coordinate (R) at (2.5, -2.0);
\draw (Q)--(R);
\draw (K)--(R);
\draw (J)--(Q);
\draw (P)--(Q);
\draw (O)--(R);
\draw (M)--(P);
\draw (O)--(N);
\draw (I)--(M);
\draw (N)--(L);

\coordinate (C1) at (1.2, -0.3);
\coordinate (C2) at (1.9, -1.4);
\draw [thick]  (C1)--(C2);


\coordinate (A) at (-6.5,-4.5);
\coordinate (B) at (-4.7,-4.5);
\coordinate (C) at (-4.7,-3.9);
\coordinate (D) at (-6.5,-3.9);
\draw [thick] (A)--(B)--(C)--(D)--cycle;

\node[draw=white] at (-6.2, -4.2) {$1$};
\node[draw=white] at (-5.6, -4.2) {$2$};

\coordinate (I) at (-6.8, -4.5);
\coordinate (J) at (-4.4,  -4.5);
\coordinate (K) at (-4.4, -5.1);
\coordinate (L) at (-6.8, -5.1);
\draw [thick] (I)--(J)--(K)--(L)--cycle;

\coordinate [thick]  (E) at (-5.9, -4.5);
\coordinate [thick]  (F) at (-5.9, -3.9);
\draw [thick]  (E)--(F);

\coordinate (G) at (-5.3, -4.5);
\coordinate (H) at (-5.3 ,-3.9);
\draw [thick]  (G)--(H);

\coordinate (M) at (-6.2, -4.5);
\coordinate (N) at (-6.2, -5.1);
\draw [thick] (M)--(N);

\coordinate (O) at (-5.6, -4.5);
\coordinate (P) at (-5.6, -5.1);
\draw [thick] (O)--(P);

\coordinate (Q) at (-5.0, -4.5);
\coordinate (R) at (-5.0, -5.1);
\draw [thick] (Q)--(R);

\draw (I)--(N);
\draw (L)--(M);
\draw (O)--(N);
\draw (P)--(M);
\draw (P)--(Q);
\draw (O)--(R);
\draw (J)--(R);
\draw (K)--(Q);


\coordinate (A) at (-3.5,-4.5);
\coordinate (B) at (-1.7,-4.5);
\coordinate (C) at (-1.7,-3.9);
\coordinate (D) at (-3.5,-3.9);
\draw [thick] (A)--(B)--(C)--(D)--cycle;

\coordinate (I) at (-3.8, -4.5);
\coordinate (J) at (-1.4,  -4.5);
\coordinate (K) at (-1.4, -5.1);
\coordinate (L) at (-3.8, -5.1);
\draw [thick] (I)--(J)--(K)--(L)--cycle;

\node[draw=white] at (-3.2, -4.2) {$4$};
\node[draw=white] at (-2.6, -4.2) {$5$};
\node[draw=white] at (-2.0, -4.2) {$6$};

\coordinate (E) at (-2.9, -4.5);
\coordinate (F) at (-2.9, -3.9);
\draw [thick] (E)--(F);

\coordinate (G) at (-2.3, -4.5);
\coordinate (H) at (-2.3 ,-3.9);
\draw (G)--(H);

\coordinate (M) at (-3.2, -4.5);
\coordinate (N) at (-3.2, -5.1);
\draw [thick] (M)--(N);

\coordinate (O) at (-2.6, -4.5);
\coordinate (P) at (-2.6, -5.1);
\draw [thick]  (O)--(P);

\coordinate (Q) at (-2.0, -4.5);
\coordinate (R) at (-2.0, -5.1);
\draw [thick] (Q)--(R);

\draw (I)--(N);
\draw (L)--(M);
\draw (O)--(N);
\draw (P)--(M);
\draw (P)--(Q);
\draw (O)--(R);
\draw (J)--(R);
\draw (K)--(Q);


\coordinate (A) at (-0.5,-4.5);
\coordinate (B) at (1.3,-4.5);
\coordinate (C) at (1.3,-3.9);
\coordinate (D) at (-0.5,-3.9);
\draw [thick] (A)--(B)--(C)--(D)--cycle;

\coordinate (I) at (-0.8, -4.5);
\coordinate (J) at (1.6,  -4.5);
\coordinate (K) at (1.6, -5.1);
\coordinate (L) at (-0.8, -5.1);
\draw [thick] (I)--(J)--(K)--(L)--cycle;

\node[draw=white] at (-0.2, -4.2) {$8$};

\coordinate (E) at (0.1, -4.5);
\coordinate (F) at (0.1, -3.9);
\draw [thick] (E)--(F);

\coordinate (G) at (0.7, -4.5);
\coordinate (H) at (0.7 ,-3.9);
\draw [thick] (G)--(H);

\coordinate (M) at (-0.2, -4.5);
\coordinate (N) at (-0.2, -5.1);
\draw [thick] (M)--(N);

\coordinate (O) at (0.4, -4.5);
\coordinate (P) at (0.4, -5.1);
\draw [thick] (O)--(P);

\coordinate (Q) at (1.0, -4.5);
\coordinate (R) at (1.0, -5.1);
\draw [thick]  (Q)--(R);

\draw (I)--(N);
\draw (L)--(M);
\draw (O)--(N);
\draw (P)--(M);
\draw (P)--(Q);
\draw (O)--(R);
\draw (J)--(R);
\draw (K)--(Q);


\node[draw=white] at (2.8, -4.2) {$10$};

\coordinate (A) at (2.5,-4.5);
\coordinate (B) at (4.3,-4.5);
\coordinate (C) at (4.3,-3.9);
\coordinate (D) at (2.5,-3.9);
\draw [thick] (A)--(B)--(C)--(D)--cycle;

\coordinate (I) at (2.2, -4.5);
\coordinate (J) at (4.6,  -4.5);
\coordinate (K) at (4.6, -5.1);
\coordinate (L) at (2.2, -5.1);
\draw [thick] (I)--(J)--(K)--(L)--cycle;

\coordinate (E) at (3.1, -4.5);
\coordinate (F) at (3.1, -3.9);
\draw [thick] (E)--(F);

\coordinate (G) at (3.7, -4.5);
\coordinate (H) at (3.7,-3.9);
\draw [thick] (G)--(H);

\coordinate (M) at (2.8, -4.5);
\coordinate (N) at (2.8, -5.1);
\draw [thick] (M)--(N);

\coordinate (O) at (3.4, -4.5);
\coordinate (P) at (3.4, -5.1);
\draw [thick] (O)--(P);

\coordinate (Q) at (4, -4.5);
\coordinate (R) at (4, -5.1);
\draw [thick] (Q)--(R);

\draw (I)--(N);
\draw (L)--(M);
\draw (O)--(N);
\draw (P)--(M);
\draw (P)--(Q);
\draw (O)--(R);
\draw (J)--(R);
\draw (K)--(Q);
\coordinate (C1) at (0, -0.3);
\coordinate (C2) at (-1.1, -1.4);
\draw (C1)--(C2);

\end{tikzpicture}
\end{center}
  \caption{A quaternary tree on sixteen keys.}
  \label{Fig:mary}
\end{figure}
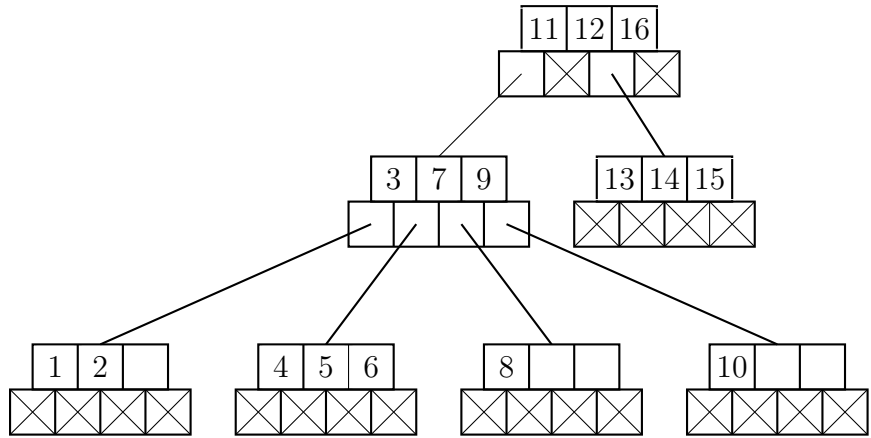

\section{Degree profile}
We characterize in this section the distribution
of nodes with a certain number of children. If a node
has $k$ children, it is of outdegree $k$.
The instrument for this analysis is a \Polya\ urn.
Several urn schemes have been proposed to study 
the nodes of an $m$--ary search trees. There is one in~\cite{Chauvin} 
(also surveyed in~\cite{Polyabook}) crafted for the study of the phase 
change after $m=26$ in the total number of nodes. The space requirement for an $m$--ary search trees for~$m \ge 27$ was studied in \cite{FillKapur} using the contraction method. 
Recently, Holmgren and Janson~\cite{Cecilia}
proposed the use of a generalized \Polya\ urn to study the so called 2-protected nodes (a more formal definition of $k$--protected nodes appears later in this section) in $m$--ary search trees. 
The model we propose
is different from the schemes mentioned. It is a bottom-up urn scheme suited for the study of outdegrees. Modifying
the model naturally requires a new eigenvalue and eigenvector 
analysis.

Let us think of the insertion positions (gaps between keys for additional insertions) 
as colored balls
in an urn. We use a color code with a collection of 
$2m-2$ colors. According to the random permutation model,
all gaps are equally likely positions for the next key insertion.
There may be gaps in nonleaf (internal) nodes. An internal node carries $m-1$
keys. If such an internal node has 
$m$ children, there are no gaps in it.
However, one with fewer children will have 
gaps represented in the data structure
by  {\bf nil} pointers. Consider a gap of this kind with $i-1$ other gaps within a 
node as a type--$i$ gap. We call such a node a node of type $i$, too. These gaps are 
insertion positions at the next level higher up in the tree
above the node containing them.

A leaf containing $i-1$ keys has $i$ gaps between its 
keys, for $i=2, \ldots, m-1$. Consider these $i$ 
gaps to be insertion positions of type--$(m+i-1)$; we represent each
type--$(m+i-1)$ gap with a ball of color $m+i-1$. 
We call such a leaf a node of type--$(m+i-1)$, too.
These gaps are 
insertion positions at the same level in the tree
as the node containing them.

The rules of evolution of the urn are as follows. 
If a ball of color $m+i-1$, for $i= 2, 3, \ldots, m-2$, is picked, the node containing the corresponding gap receives a key. 
The number of keys
in the node goes up by 1, and the number of gaps also goes up by 1.
So, we remove $i$ balls of color $m+i-1$ 
from the urn and add to it $i+1$ balls of color $m+i$.
The rules for balls of the other colors are different.
If a ball (gap) of color $i$ is picked ($i = 2, \ldots, m$), 
a key is inserted in a leaf at the next level in the tree above the node containing the gap.
Such a new leaf contains 1 key (two gaps), i.e., we should remove~$i$ balls of color $i$, add $i-1$ balls of color $i-1$ 
and add two balls of color $m+1$. 
If a ball of color $1$ is picked, 
a key is inserted in a leaf at the next level in the tree above the node containing the gap, no other gaps are left in the node.
Such a new leaf contains one key (two gaps), i.e., we should remove one balls of color $1$, 
and add two balls of color $m+1$.
Lastly, if a gap of color $2m-2$ is picked, the leaf containing them receives a new key and fills out, and $m$ insertion positions appear at the next level of insertion. We remove 
$(m-1)$ balls of color $2m-2$ and add $m$ balls of color $m$ to the urn.

We shall analyze the number $X_n^{(i)}$, for $i = 1, \ldots, 2m-2,$
of gaps of type~$i$ after $n$ insertions.
The tree in Figure~\ref{Fig:mary} has 
$X_n^{(1)} = 0$,
$X_n^{(2)} = 2$,
$X_n^{(3)} = 0$,
$X_n^{(4)} = 8$,
$X_n^{(5)} = 4$, and
$X_n^{(6)} = 3$.

We can represent this ball addition scheme by a replacement matrix $\matA$, the rows
and columns of which are indexed from 1 to $2m-2$. We thus have
$$\matA = \begin{pmatrix}
             -1      &0   &0&\ldots   &0&2&0&0&\ldots&0\\
              1  &-2 &0&\ldots&0&2&0&0&\ldots&0\\
            0      &2   &-3    &\ldots&0&2&0&0&\ldots&0\\
\vdots      &\vdots   &\vdots    &&\vdots&\vdots&\vdots&\vdots&&\vdots\\
 0      &0   &0    &\ldots&0&2&0&0&\ldots&0\\
 0      &0   &0    &\ldots&0&2&0&0&\ldots&0\\
 0      &0   &0    &\ldots&-m&2&0&0&\ldots&0\\
             0      &0   &0    &\ldots&0&-2&3&0&\ldots&0\\
              0      &0   &0    &\ldots&0&0&-3&4& \ldots&0\\
              0      &0   &0    &\ldots&0&0&0&-4& \ldots&0\\
\vdots      &\vdots   &\vdots    &&\vdots&\vdots&\vdots&\vdots&&\vdots\\
                0      &0   &0    &\ldots&0&0&0&0&\ldots&0\\
                0      &0   &0    &\ldots&0&0&0&0&\ldots&(m-1)\\
                 0      &0   &0   &\ldots&m&0&0&0&\ldots&-(m-1)
             \end{pmatrix}. $$ 
The binary tree ($m=2$) is a boundary case with the replacement matrix
$\begin{pmatrix} - 1 &2 \\
                            1 &0\end{pmatrix}$.
              
A successful asymptotic analysis of this urn scheme relies on the 
eigenvalues and eigenvectors, which we take up next. Suppose
we organized the~$m$ eigenvalues of $\matA^T$ (here $T$
denotes the transpose) according to their descending values 
of the real part, i.e., they are of the form
$$\Re\, \lambda_1 \ge \Re\,\lambda_2\ge \ldots \ge \Re\, \lambda_{m-1}.$$
Hence, $\lambda_1$ is the \emph{principal eigenvalue}, the one with the largest
real part. We call the corresponding eigenvector $\vecv$ the 
\emph{principal eigenvector}.

Toward a first-order analysis (strong laws) we only need
the principal eigenvalue and principal eigenvector of $\matA^T$. 
The urn scheme we have is balanced, with constant row sum (equal to 1).
According to a theorem of Perron and 
Frobenius~\cite{Frobenius,Perron}, $\lambda_1 = 1$.
\begin{lem}
\label{Lem:principal}
The principal eigenvector of $\matA^T$ is
$$\vecv =    \begin{pmatrix}
            v_1 \\
            v_2\\
               \vdots\\
                 v_{2m-2}\\
             \end{pmatrix}= \frac 1 {m(m+1)(H_m -1)}\begin{pmatrix}
         1\\
          2\\
         \vdots \\
         m\\
        \frac 1 3 m (m+1)\\
        \frac 1 4 m (m+1)\\
               \vdots\\
              \frac 1 m m (m+1)\\
             \end{pmatrix}. $$
\end{lem}
\begin{proof}
We should find the principle eigenvector with length 1. 
That is, we are solving under the constraint 
$\sum_{i=1}^m v_i = 1$.  To determine the principal eigenvector we solve the equation
$\matA^T \vecv= \vecv$. The first $m$ rows give the equations
$$-iv_i + i v_{i+1}= v_i, \qquad i = 1, \ldots, m.$$
Thus, $v_i = i v_1$, for $i = 1, \ldots, m$.
Row $m+1$ gives the equation
$$2v_1 + 2v_2 + \cdots+2 v_m - 2v_{m+1} = v_{m+1}.$$
The following rows give
the equations
$$(i+1)v_{m+i} - (i+1) v_{m+i+1}= v_{m+i}, \qquad \mbox{for\ } i = 2, \ldots, m-2.$$
Solving the equations under the aforementioned length 
constraint, we get
the stated solution.  
\end{proof} 
We appeal to a classic urn result by Athreya and Karlin~\cite{Athreya} 
(see also~\cite{Janson,Smythe}) that relates
the proportion of balls of various colors to the principal eigenvalue and eigenvector. Specialized to our case,
this result translates in the following.
\begin{theorem} 
\label{Theo:Athreya}
Let $X_n^{(i)}$ be the number of gaps of type~$i$ (color $i$) in an
$m$--ary search tree grown from a random permutation, for $i= 1, 2,  \ldots, 2m-2$. 
Let~$\vecX_n$ be the vector with these components. 
We then have\footnote{For random vectors ${\bf Y}_n$ and $\bf Y$,
the notation ${\bf Y}_n \almostsure {\bf Y}$ is for component-wise 
almost sure convergence.}
$$\frac 1 n \, \vecX_n \almostsure \vecv. $$
\end{theorem}
As a corollary of this theorem, we recover a recent 
result derived in~\cite{Cecilia}.
\begin{cor}  (Holmgren and Janson~\cite{Cecilia}) 
\label{Cor:Cecilia}
Let $L_n$ be the number of leaves in the
$m$--ary search tree grown from a random permutation. We then have
$$\frac 1 n L_n \almostsure \frac {m-1} {2(m+1)(H_m-1)}.$$ 
\end{cor}
\begin{proof}
For $i=2, \ldots, m-1$,  groups of size $i$ of  type--$(m+i-1)$ gaps
reside in one node and account for one leaf.  
Also, groups of size $m$ of type--$m$ gaps
account for one leaf. Thus, we have
$$L_n = \frac 1 m X_n^{(m)} + \sum_{i=2}^{m-1} \frac {X_n^{(m+i-1)}} i.$$
Hence, $L_n/ n$ converges almost surely to $v_m/m + \sum_{i=2}^{m-1} v_{m+i-1}/i$.
The statement follows from Lemma~\ref{Lem:principal} after simplifying the algebra.
\end{proof}
Another corollary concerns protected nodes. A node in a tree
is said to be~$k$ protected, if the nearest leaf is at distance 
$k$ (edges). Recently, the interest has surged in protected nodes
in various families of trees (particularly 2--protected nodes); 
see~\cite{Cheon,Du,Cecilia,WardBST,Wardrec,Mansour} for
the counterpart in ordered trees, digital trees, $m$--ary search trees,
binary search trees, recursive trees, and uniform $m$--ary trees, respectively.
A broad view of protected nodes 
covering many tree classes is in~\cite{DevroyeJanson}.
\begin{cor} (Holmgren and Janson~\cite{Cecilia}) 
Let $P_n$ be the number of 1--protected nodes in an
$m$--ary search tree grown from a random permutation. We then have
$$\frac 1 n P_n \almostsure \frac 1 {2(m+1)(H_m-1)}.$$ 
\end{cor}
\begin{proof}
Let $S_n$ be the number of nodes in the tree.
It is shown in~\cite{Chauvin} that $S_n / n \almostsure  1/ (2(H_m-1))$. The 1--protected nodes are the nonleaves.
Their number is $P_n = S_n - L_n$.
\end{proof}
Another corollary addresses the title of the paper. There is
quite a bit of interest in studying degree profiles of trees;
see~\cite{Drmota,JansonRSA, Kuba}, for example.

The results for $X_n^{(i)}$ can help us find a profile of outdegrees.
Let $D_n^{(k)}$ be the number of nodes of outdegree $k$,
for $k= 0, \ldots, m$.
The tree in Figure~\ref{Fig:mary} has
$D_n^{(0)} = 5$,
$D_n^{(1)} = 0$,
$D_n^{(2)} = 1$,
$D_n^{(3)} = 0$, and
$D_n^{(4)} = 1$.
\begin{cor} 
\label{Cor:degrees}
Let $D_n^{(k)} $ be the number of nodes of outdegree $k$ in an
$m$--ary search tree grown from a random permutation. We then have
\begin{align*}
\frac 1 n D_n^{(0)} &\almostsure \frac {m-1} {2(m+1)(H_m-1)},\\
\frac 1 n D_n^{(k)} &\almostsure \frac 1 {m(m+1)(H_m-1)}, \qquad \mbox{for \ } k =1 , \ldots, m-1,\\
\frac 1 n D_n^{(m)} &\almostsure \frac 1{m(m+1)(H_m-1)}.
\end{align*} 
\end{cor}
\begin{proof}
The number of leaves (nodes of outdegree 0) is derived in~\cite{Cecilia},
as discussed in Corollary~\ref{Cor:Cecilia}.
For $i = 1, \ldots, m-1$, every node of outdegree $i$ has $m-i$ gaps of type
$m-i$, and thus we have
$$D_n^{(i)} = \frac {X_n^{(m-i)}} {m-i},$$
and $D_n^{(i)} / n \to \lim_{n\to\infty} \frac {X_n^{(m-i)}} {(m-i)n} = v_{m-i}/(m-i)$,
where $v_{m-i}$ is the $(m-i)$th component of the principal eigenvector 
(cf.\ Lemma~\ref{Lem:principal}).
There is one last node type we have not accounted for with balls in the urn.
These are the full nodes, where every key slot is taken and every pointer is used.
Full nodes  are
of outdegree $m$. These are all the nodes (total $S_n$) excluding nodes of outdegree less than $m$,
and their count is
$$D_n^{(m)} = S_n - \sum_{i=0}^{m-1} D_n^{(i)}.$$
Their limiting proportion is therefore
$$\lim_{n\to\infty}\frac {D_n^{(m)}} n  = \lim_{n\to\infty} 
    \Bigl(\frac {S_n} n - \frac {L_n} n - \sum_{i=1}^{m-1} \frac{D_n^{(i)}}n \Bigr) .$$
We get the limiting proportion $S_n/n$ from~\cite{Chauvin}, and obtain the other elements
of the calculation from Corollary~\ref{Cor:Cecilia} and the already established parts of this proof.
\end{proof}
A main result of this investigation is that the count $D_n^{(k)}$ (for $k=0, \ldots, m$)
has a phase change after $m=26$. For values of $m$
up to 26, the joint distribution of the number of nodes
of various outdegrees is asymptotically multivariate normal in $m$ dimensions.\footnote{The multivariate normal distributions we refer to is what some books call singular
multivariate normal distributions, where $\bf{\Sigma}$ is a singular matrix, but one can find a nonempty subset of these random variables of size $k<m$, that jointly have together a proper multivariate normal distribution of $k$ dimensions.} 
At $m =  27$, the asymptotic distribution is not Gaussian. Such curious phase change was detected in the context of $m$--ary searh trees before, and the number 26 was also the threshold found in~\cite{Pittel}, where the authors analyzed the total amount of space allocated to  a random $m$--ary search trees. 

We shall use $\normal_m({\bf 0}, {\bf \Sigma})$ to denote a centered 
multivariate normal (possibly improper) vector in $m$ dimensions with an $m\times m$ covariance matrix~$\bf \Sigma$, and mean $\bf 0$ (of $m$ components, all 0).
This result holds in view of known urn theory~\cite{Janson,Smythe}, according
to which the number of balls of colors $j$, $j= 1, \ldots, c$, in a $c$--color urn
with certain properties, when appropriately normalized, converges in distribution to a multivariate normal. The theorem needs several conditions, among
which $\Re \, \lambda_2 < \frac 1 2 \Re\, \lambda_1$. All the conditions
hold in our case, except the latter eigenvalue requirement, which holds only
up to~26. The tables below display $\Re\, \lambda_2$ 
(approximated to three decimal 
places), for $m$ up to $27$.
These tables show how $\Re\, \lambda_2$ steadily increases
with~$m$, staying below~$\frac 1 2 \lambda_1 = \frac 1 2$, till $m=26$; at $m=27$
the mid-mark point is crossed.

\bigskip
\begin{sideways}
\begin{tabular}{|>{\begin{turn}{-90}}c<{\end{turn}}|>{\begin{turn}{-90}}c<{\end{turn}}|}
\hline
$\Re\,\lambda_2$ & $m$\\ \hline
$-2.000$ & 2\\ \hline
$-2.000$ & 3\\ \hline
$-2.000$ & 4\\ \hline
$-2.000$ & 5\\ \hline
$-1.768$ & 6\\ \hline
$-1.260$ & 7\\ \hline
$-0.899$ & 8\\ \hline
\end{tabular}
\end{sideways}

\bigskip
\begin{sideways}
\begin{tabular}{|>{\begin{turn}{-90}}c<{\end{turn}}|>{\begin{turn}{-90}}c<{\end{turn}}|}
\hline
$\Re\,\lambda_2$ & $m$\\ \hline
$-0.633$ & 9\\ \hline
$-0.431$ & 10\\ \hline
$-0.273$ & 11\\ \hline
$-0.147$ & 12\\ \hline
$-0.044$ & 13\\ \hline
0.040 & 14\\ \hline
0.112 & 15\\ \hline
0.173 & 16\\ \hline

\end{tabular}
\end{sideways}

\bigskip
\begin{sideways}
\begin{tabular}{|>{\begin{turn}{-90}}c<{\end{turn}}|>{\begin{turn}{-90}}c<{\end{turn}}|}
\hline
$\Re\,\lambda_2$ & $m$\\ \hline
0.226 & 17\\ \hline
0.272 & 18\\ \hline
0.313 & 19\\ \hline
0.348 & 20\\ \hline
0.380 & 21\\ \hline
0.409 & 22\\ \hline
0.435 & 23\\ \hline
0.458 & 24\\ \hline
0.479 & 25\\ \hline
\end{tabular}
\end{sideways}

\bigskip
\begin{sideways}
\begin{tabular}{|>{\begin{turn}{-90}}c<{\end{turn}}|>{\begin{turn}{-90}}c<{\end{turn}}|}
\hline
$\Re\,\lambda_2$ & $m$\\ \hline
0.499 & 26\\ \hline
0.516 & 27\\ \hline
\end{tabular}
\end{sideways}

In formulating a result for the limiting distribution 
we appeal to~\cite{Smythe}, where it is argued that a central limit theorem holds for a \Polya\ urn scheme with replacement matrix $\tilde \matA$, if certain conditions are met. 
The class discussed in~\cite{Smythe} includes
random replacement matrices, where a random number of balls of a certain color may be added.  
We list these conditions here, so that our presentation is self contained:
\begin{itemize}
\item [(a)] The urn scheme is tenable---it is possible to draw balls indefinitely on every stochastic path. 
\item [(b)] The average replacement matrix $\E[\tilde \matA]$ has constant row sums.
\item [(c)] Every entry in $\E[\tilde \matA]$ has a finite second moment.
\item [(d)] The matrix $\E[\tilde \matA]$  has a unique maximal (principal) positive eigenvalue of multiplicity 1 with strictly positive left eigenvector (in all its components).
\end{itemize}
The class of urn schemes meeting these conditions  is called Extended \Polya\ Urns.
Note that our replacement matrix is extended and
meets all these conditions.

Define $\vecv^*$  as the limit of $\vecD_n/ n$.
Namely, by Corollary~\ref{Cor:degrees}, it is

\begin{align*}
{\bf v^*}  =
\frac 1 {(m+1)(H_m-1)}\begin{pmatrix}
 \frac 1 2 (m-1)\\
 \frac 1 {m}\\
 \frac 1 {m}\\
  \vdots\\
 \frac 1{m}
 \end{pmatrix}.
\end{align*} 
\begin{theorem} 
Let $D_n^{(i)}$ be the number of nodes of outdegree $i$ in an
$m$--ary search tree grown from a random permutation, for $i=0, 1, \ldots, m$. Let $\vecD_n$ be the vector with these components.
For $m=3, \ldots, 26$, we have, as $n\to\infty$,  
$$\frac{\vecD_n - n\vecv^*} {\sqrt n}\ \convD\ \normal_m(0, {\bf \Sigma}),$$
for an effectively computable covariance matrix ${\bf \Sigma}$. The covariance matrix can be computed from formulas in~\cite{Janson}.
\end{theorem}
\section{Practical implications}
To implement an $m$--ary tree in practice, one uses blocks of 
data (records) to represent the nodes. Each node should have
$m-1$ data slots and $m$ pointers (places to hold the memory addresses
of records containing the children of that node). If a certain child record does not
exist, the pointer to that child contains a {\bf nil} value that
points to nowhere specific, indicating an empty subtree. 
For instance, a ternary tree ($m=3$) designed
to hold real numbers can be crafted from two data slots for keys,
and three pointers. In pseudo code this might look like

\bigskip
\vbox { \baselineskip=12pt\obeylines        
           { {\bf Type}        
              \qquad     pointer = $\uparrow$node;        
              \qquad     node = {\bf record}        
        \qquad      \qquad         smallernumber: {\bf real};  
       \qquad      \qquad         largerernumber: {\bf real};       
        \qquad      \qquad         left:   {\bf pointer};        
        \qquad      \qquad         middle:  {\bf pointer};        
        \qquad      \qquad         right:   {\bf pointer}        
        \qquad {\bf end};\bigskip}}  

On a typical small modern computer (such as a desktop PC or a portable Mac)  
single-precision real data are represented 
in one word 
(4 bytes, 8 bits each) and pointers may be four bytes each. The random access memory (RAM) space is 4 Gigabytes ($2^{32}$ bytes), and 4 bytes will specify any address in the RAM.  

The analysis we went through shows possibilities for data structure compression.
We can think of a data structure with multiple types of nodes.
Full nodes with $m$ children are of the usual type ($m-1$ keys and $m$ pointer).
However, saving memory space can take effect in nodes
with fewer than $m$ children.
We can eliminate~$m$ pointers from each leaf of type--$(m+i-1)$ node, 
for $i=2, \ldots, m-1$,  and eliminate
$m-i$ key slots. The gaps 
of color~$m$ also fall in this category.  

As we have many node types, we need a \emph{descriptor}
in each node to tell its type, which we suppose needs $\delta_m$ bytes. 
This descriptor should be long enough to distinguish $2m-2$ types of nodes.
For instance, with $m$ up to~256, one byte ($\delta_m=1$) on any modern computer is sufficient to encode all the node types.
In what follows we suppose each pointer needs $p$ bytes, and each key needs $k$ bytes. 

A leaf containing $i$ gaps
has $i-1$ keys,  and no children. These are the gaps of  colors $m, m+1, \ldots, 2m-2$. The 
pointers in the usual implementation are superfluous and can be eliminated. Such a node can be presented in a frugal leaf containing only space 
for $i-1$ keys and the descriptor. 

An internal node with children but fewer than $m$ 
has unused pointers, but
all the key slots are occupied.
For nodes of type $i$ (for $i = 1, \ldots, m-1$), we can eliminate $i$ pointers.
However, for insertion, searching and data retrieval algorithms to operate correctly, an algorithm reaching such a node should know where
the nonnull pointers are (which ones are occupied). This needs a bitmap (a secondary
descriptor) of length $m$ bits of ones and zeros, with zeros corresponding
to $\bf nil$ pointers, and ones corresponding to actively engaged nexuses pointing
to nonempty subtrees.  We suppose the bitmap needs $\Delta_m$ bytes. 

There is one last node type that needs more space---the full nodes. There is $D_n^{(m)}$ of them and they do not need a bitmap.
Let us consider these as nodes of an additional type: type--$(2m-1)$. 
The space needed for these full nodes is 
$$\bigl(\delta_m + (m-1) k + m p\bigr)  {D_n^{(m)}}. $$
\usetikzlibrary{calc,intersections}

The compression scheme is illustrated in Figure~\ref{Fig:compactmary}.
This is the same tree of Figure~\ref{Fig:mary} with all {\bf nil} pointers and unused 
key slots removed. Notice that some nodes in the compact 
tree are tagged on the left with
numbers in circles. The circle is the space required for the primary descriptor (needed
for all the nodes). In this example, $m=4$ and we have $2m-1 = 7$ node types. One byte is sufficient to represent whole numbers up to 7. For instance, the root node in
Figure~\ref{Fig:mary} has two unused pointers. So, it is of type 2, the number used in the one-byte left tag of the root node of the compact tree in Figure~\ref{Fig:compactmary}. The leftmost node on level one in the tree is a full node (type--7), we need to keep all its components 
in the compact tree, and tag it with 7 on the left. Nodes of types $1, \ldots, m-1$ are further 
tagged with a secondary descriptor to store the bitmap. 
In the compact tree of Figure~\ref{Fig:compactmary}, we only have
one node of the first three types (the root node). Only the leftmost pointer and the third from  the left need to be preserved; the bitmap is 1010. Four bits can be represented in one byte. This number is 10 in decimal, which is the number in the one-byte descriptor tagging the root on the right side.

\begin{figure}[thb]
\begin{center}
\begin{tikzpicture}
\coordinate (A) at (0,0);
\coordinate (B) at (1.8,0);
\coordinate (C) at (1.8,0.6);
\coordinate (D) at (0,0.6);
\draw [thick] (A)--(B)--(C)--(D)--cycle;

\node[draw=white] at (0.3, 0.3) {$11$};
\node[draw=white] at (0.9, 0.3) {$12$};
\node[draw=white] at (1.5, 0.3) {$16$};

\node[draw=white] at (-0.3, 0.3) {$2$};
\draw  [thick](-0.3,0.3) circle (8pt);

\node[draw=white] at (2.1, 0.3) {$10$};
\draw  [thick](2.1,0.3) circle (8pt);

\coordinate (E) at (0.6, 0);
\coordinate (F) at (0.6, 0.6);
\draw [thick] (E)--(F);

\coordinate (G) at (1.2, 0);
\coordinate (H) at (1.2, 0.6);
\draw [thick] (G)--(H);

\coordinate (I) at (-0.3,0);
\coordinate (J) at (0.3,0);
\coordinate (K) at (0.3,-0.6);
\coordinate (L) at (-0.3,-0.6);
\draw [thick] (I)--(J)--(K)--(L)--cycle;

\coordinate (M) at (0.3, 0);
\coordinate (N) at (0.3, -0.6);
\draw [thick]  (M)--(N);

\coordinate (O) at (0.9, 0);
\coordinate (P) at (0.9, -0.6);
\draw [thick] [thick]  (O)--(P);

\coordinate (Q) at (1.5, 0);
\coordinate (R) at (1.5, -0.6);
\draw [thick]  (Q)--(R);
\draw [thick] (P)--(R);


\coordinate (A) at (-2.0,-2);
\coordinate (B) at (-0.2,-2);
\coordinate (C) at (-0.2,-1.4);
\coordinate (D) at (-2.0,-1.4);
\draw [thick] (A)--(B)--(C)--(D)--cycle;

\coordinate (I) at (-2.3,-2);
\coordinate (J) at (0.1,-2);
\coordinate (K) at (0.1,-2.6);
\coordinate (L) at (-2.3,-2.6);
\draw [thick] (I)--(J)--(K)--(L)--cycle;

\node[draw=white] at (-1.7, -1.7) {$3$};
\node[draw=white] at (-1.1, -1.7) {$7$};
\node[draw=white] at (-0.5, -1.7) {$9$};

\node[draw=white] at (-2.3, -1.7) {$7$};
\draw  [thick](-2.3,-1.7) circle (8pt);

\coordinate (E) at (-1.4, -2);
\coordinate (F) at (-1.4, -1.4);
\draw [thick] (E)--(F);

\coordinate (G) at (-0.8, -2);
\coordinate (H) at (-0.8 ,-1.4);
\draw [thick]  (G)--(H);

\coordinate (M) at (-1.7, -2.6);
\coordinate (N) at (-1.7, -2.0);
\draw [thick]  (M)--(N);

\coordinate (O) at (-1.1, -2.6);
\coordinate (P) at (-1.1, -2.0);
\draw [thick] (O)--(P);

\coordinate (Q) at (-0.5, -2.6);
\coordinate (R) at (-0.5, -2.0);
\draw [thick] (Q)--(R);


pointers from the node
\coordinate (C1) at (-2, -2.3);
\coordinate (C2) at (-5.6, -3.9);
\draw [thick] (C1)--(C2);

\coordinate (C3) at (-1.4, -2.3);
\coordinate (C4) at (-2.6, -3.9);
\draw [thick] (C3)--(C4);

\coordinate (C5) at (-0.8, -2.3);
\coordinate (C6) at (0.4, -3.9);
\draw [thick] (C5)--(C6);

\coordinate (C7) at (-0.2, -2.3);
\coordinate (C8) at (3.4, -3.9);
\draw [thick]  (C7)--(C8);

\node[draw=white] at (0.7, -1.7) {$4$};
\draw  [thick](0.7,-1.7) circle (8pt);

\coordinate (A) at (1.0,-2);
\coordinate (B) at (2.8,-2);
\coordinate (C) at (2.8,-1.4);
\coordinate (D) at (1.0,-1.4);
\draw [thick] (A)--(B)--(C)--(D)--cycle;

\node[draw=white] at (1.3, -1.7) {$13$};
\node[draw=white] at (1.9, -1.7) {$14$};
\node[draw=white] at (2.5, -1.7) {$15$};

\coordinate (E) at (1.6, -2);
\coordinate (F) at (1.6, -1.4);
\draw [thick] (E)--(F);

\coordinate (G) at (2.2, -2);
\coordinate (H) at (2.2 ,-1.4);
\draw [thick]  (G)--(H);

\coordinate (C1) at (1.2, -0.3);
\coordinate (C2) at (1.9, -1.4);
\draw [thick]  (C1)--(C2);

\node[draw=white] at (-6.5, -4.2) {$6$};
\draw  [thick](-6.5,-4.2) circle (8pt);

\coordinate (A) at (-6.2,-4.5);
\coordinate (B) at (-5.0,-4.5);
\coordinate (C) at (-5.0,-3.9);
\coordinate (D) at (-6.2,-3.9);
\draw [thick] (A)--(B)--(C)--(D)--cycle;

\node[draw=white] at (-5.9, -4.2) {$1$};
\node[draw=white] at (-5.3, -4.2) {$2$};

\coordinate [thick]  (E) at (-5.6, -4.5);
\coordinate [thick]  (F) at (-5.6, -3.9);
\draw [thick]  (E)--(F);


\node[draw=white] at (-3.8, -4.2) {$4$};
\draw  [thick](-3.8,-4.2) circle (8pt);

\coordinate (A) at (-3.5,-4.5);
\coordinate (B) at (-1.7,-4.5);
\coordinate (C) at (-1.7,-3.9);
\coordinate (D) at (-3.5,-3.9);
\draw [thick] (A)--(B)--(C)--(D)--cycle;

\node[draw=white] at (-3.2, -4.2) {$4$};
\node[draw=white] at (-2.6, -4.2) {$5$};
\node[draw=white] at (-2.0, -4.2) {$6$};

\coordinate (E) at (-2.9, -4.5);
\coordinate (F) at (-2.9, -3.9);
\draw [thick] (E)--(F);

\coordinate (G) at (-2.3, -4.5);
\coordinate (H) at (-2.3 ,-3.9);
\draw (G)--(H);

\node[draw=white] at (-0.2, -4.2) {$5$};
\draw  [thick](-0.2,-4.2) circle (8pt);

\coordinate (A) at (0.1,-4.5);
\coordinate (B) at (0.7,-4.5);
\coordinate (C) at (0.7,-3.9);
\coordinate (D) at (0.1,-3.9);
\draw [thick] (A)--(B)--(C)--(D)--cycle;

\node[draw=white] at (0.4, -4.2) {$8$};

\node[draw=white] at (3.4, -4.2) {$10$};
\node[draw=white] at (2.8, -4.2) {$5$};
\draw  [thick](2.8,-4.2) circle (8pt);

\node[draw=white] at (3.4, -4.2) {$10$};

\coordinate (A) at (3.1,-4.5);
\coordinate (B) at (3.7,-4.5);
\coordinate (C) at (3.7,-3.9);
\coordinate (D) at (3.1,-3.9);
\draw [thick] (A)--(B)--(C)--(D)--cycle;

\coordinate (E) at (3.1, -4.5);
\coordinate (F) at (3.1, -3.9);
\draw [thick] (E)--(F);

\coordinate (G) at (3.7, -4.5);
\coordinate (H) at (3.7,-3.9);
\draw [thick] (G)--(H);

\coordinate (C1) at (0, -0.3);
\coordinate (C2) at (-1.1, -1.4);
\draw (C1)--(C2);

\end{tikzpicture}
\end{center}
  \caption{A compact quaternary tree on sixteen keys.}
  \label{Fig:compactmary}
\end{figure}
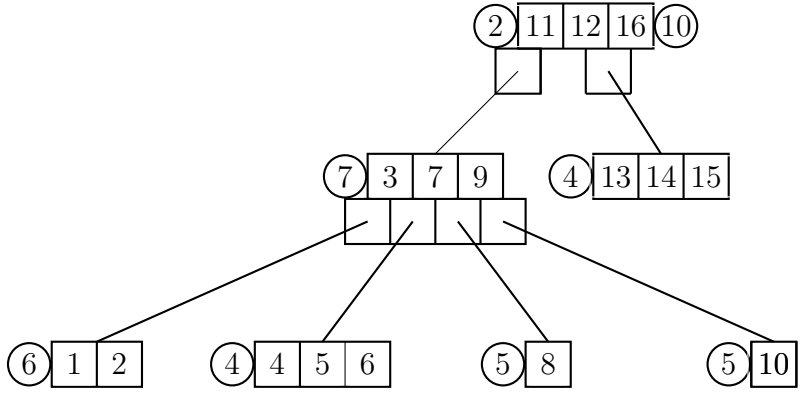

The actual physical size (in bytes) of the modified tree is
\begin{align*}
S_n' &:=   \bigl(\delta_m + (m-1) k + mp)  {D_n^{(m)}}   \\
      &\qquad {} + \sum_{i=1}^{m-1} \bigl(\delta_m
                 \Delta_m  +(m-1)k + (m-i) p\bigr)\frac{ X_n^{(i)}} i \\
        &\qquad {} + \bigl(\delta_m+ (m-1) k  \bigr) \, \frac {X_n^{(m)}} m + \sum_{i=2}^{m-1} \bigl(\delta_m+ (i-1) k  \bigr) \, \frac {X_n^{(m+i-1)}} i.
\end{align*}
On the other hand, a plain implementation, allocating a total of $S_n$
nodes (recall $S_n$ in the proof of  Corollary~\ref{Cor:Cecilia}) would use $(mp + (m-1)k)S_n$ bytes.
The relative size of the proposed modified frugal structure is then
\begin{align*}
\frac {S_n'} {(mp+(m-1)k)S_n} &\almostsure    \frac 1 {2m(m+1)(H_m-1)}
      (2m^2 k H_{m}+m^2 \delta_m\\
    &\qquad\qquad {}-2m^2 k+m^2 p+2mk H_{m}+m \delta_m+2m \Delta_m\\
    &\qquad\qquad {}+mp-2mk-2\Delta_m) \Big / \Bigl(\frac{mp+(m-1)k} {2(H_m-1)} \Bigr).
\end{align*} 

Recall that $\delta_m$ should be big enough to encode $2m-2$ node types. 
For that we need
at least $\log_2 (2m-2)$ bits. 
Computers cannot access individual bytes, and we must round up
the primary descriptor's space to the nearest number of bytes, 
which is $\delta_m = \lceil\frac 1 b \log_2 (2m-2)\rceil$, where
$b$ is the number of bits per byte.
The secondary descriptor $\Delta_m$ is a bitmap of length $m$,
and thus needs $\Delta_m = \lceil m / b\rceil$ bytes. 
Plugging in these descriptors' values and simplifying, we see that the relative size satisfies
$$\frac {S_n'} {(mp+(m-1)k)S_n} \almostsure \frac {(2k+b)\ln m }{(k+p) m},$$
offering a considerable saving for large $m$. 

In the small computer environment mentioned above, $k=4$, $p = 4$, and~$b=8$. 
The table below lists the relative asymptotic size (approximated to three decimal places)
of the modified tree to the asymptotic 
size under a plain implementation. The figures show how quickly the relative size
comes down. For instance, at $m=10$, the space saving is about $73\%$. 

In conclusion, we remark that by construction the reduction scheme is the best possible space saving that could be achieved, and that the reduction is achieved on average for all values of $m\ge 2$, such as the typical values in the hundreds commonly used in large database applications. However, for values of~$m$ greater than 26 the variability may be too much to predict that near-average saving is attained most of the time. 

\bigskip
\begin{sideways}
\begin{tabular}{|>{\begin{turn}{-90}}c<{\end{turn}}|>{\begin{turn}{-90}}c<{\end{turn}}|}
\hline
Relative size & $m$\\ \hline
0.778 & 2\\ \hline
0.600 & 3\\ \hline
0.499& 4\\ \hline
0.431& 5\\ \hline
0.383& 6\\ \hline
0.345& 7\\ \hline
0.316& 8\\ \hline
0.294& 9\\ \hline
\end{tabular}
\end{sideways}

\bigskip
\begin{sideways}
\begin{tabular}{|>{\begin{turn}{-90}}c<{\end{turn}}|>{\begin{turn}{-90}}c<{\end{turn}}|}
\hline
Relative size & $m$\\ \hline
0.273& 10\\ \hline
0.256& 11\\ \hline
0.240& 12\\ \hline
0.227& 13\\ \hline
0.215& 14\\ \hline
0.205& 15\\ \hline
0.196& 16\\ \hline
0.188& 17\\ \hline
\end{tabular}
\end{sideways}

\bigskip
\begin{sideways}
\begin{tabular}{|>{\begin{turn}{-90}}c<{\end{turn}}|>{\begin{turn}{-90}}c<{\end{turn}}|}
\hline
Relative size & $m$\\ \hline
0.180 & 18\\ \hline
0.173 & 19\\ \hline
0.167 & 20\\ \hline
0.161 & 21\\ \hline
0.156 & 22\\ \hline
0.151 & 23\\ \hline
0.146 & 24\\ \hline
0.142 & 25\\ \hline
\end{tabular}
\end{sideways}

\bigskip
\begin{sideways}
\begin{tabular}{|>{\begin{turn}{-90}}c<{\end{turn}}|>{\begin{turn}{-90}}c<{\end{turn}}|}
\hline
Relative size & $m$\\ \hline
0.138 & 26\\ \hline
0.134 & 27\\ \hline
\end{tabular}
\end{sideways}
\section*{Acknowledgments} We thank Cecilia Holmgren and Svante Janson for several valuable remarks and helpful comments.

\end{document}